\documentclass[11pt]{amsart}
\usepackage{amsmath,amsthm,amssymb,hyperref,fancyhdr,mathtools}
\usepackage{derivative}
\usepackage{color}
\usepackage{comment}
\usepackage{enumitem}
\usepackage{fullpage}
\usepackage{derivative}

\theoremstyle{plain}
\newtheorem{theorem}{Theorem}
\newtheorem{lemma}[theorem]{Lemma}
\newtheorem{proposition}[theorem]{Proposition}

\newtheorem{definition}[theorem]{Definition}

\newtheorem{remark}[theorem]{Remark}
\numberwithin{theorem}{section}
\numberwithin{equation}{section}

\newcommand{\ip}[2]{\ensuremath{\left\langle#1,#2\right\rangle}}

\newcommand{\B}{\mathbb{B}}
\newcommand{\C}{\mathbb{C}}

\newcommand{\R}{\mathbb{R}}

\newcommand{\T}{\mathbb{T}}

\newcommand{\cA}{\mathcal{A}}

\newcommand{\cF}{\mathcal{F}}

\newcommand{\cL}{\mathcal{L}}
\newcommand{\cS}{\mathcal{S}}

\newcommand{\tr}{\mathrm{tr}}

\newcommand{\id}{\mathrm{id}}

\newcommand{\Berg}{\cA^2_\nu(\Omega)} 

\newcommand{\bddf}{L^\infty(\Omega)} 
\newcommand{\bddfH}{L^\infty(\Omega)^H}
\newcommand{\bddfG}{L^\infty(G)} 
\newcommand{\Lone}{L^1(\Omega,d\lambda)} 
\newcommand{\Lp}{L^p(\Omega,d\lambda)} 
\newcommand{\LoneG}{L^1(\G)} 

\newcommand{\bdd}{\cL(\cA^2_\nu)}
\newcommand{\bddH}{\cL(\cA^2_\nu)^H}

\newcommand{\traceop}{\cS^1(\cA^2_\nu)}

\newcommand{\toepalg}{\mathfrak{T}_\nu(L^\infty)}
\newcommand{\toepalgH}{\mathfrak{T}_\nu(L^\infty)^H}

\newcommand{\wloc}{\mathcal{W}_{loc}^\alpha(\cA^2_\nu)}

\newcommand{\toep}{\mathcal{T}_\nu(L^\infty)}
\newcommand{\haar}[1]{\ d\mu_{G}(#1)}

\newcommand{\inv}[1]{\ d\lambda(#1)}

\newcommand{\dvz}[1]{\ dv_\nu(#1)}
\newcommand{\dv}{dv_\nu}
\newcommand{\actopL}[2]{L_{#1}^\nu(#2)}
\newcommand{\actfL}[2]{\ell_{#1}(#2)}

\newcommand{\act}[2]{\tau_{#1}(#2)}

\newcommand{\G}{G}

\newcommand{\Br}{\boldsymbol{G_\rho}}
\newcommand{\Dr}{B_\rho}
\newcommand{\Brr}[1]{\boldsymbol{G_{#1}}}

\newcommand{\pinu}{\pi_\nu}
\newcommand{\knu}{k^\nu}
\newcommand{\ipnu}[2]{\ensuremath{\left\langle#1,#2\right\rangle}_\nu}
\newcommand{\normnu}[1]{\|#1\|_\nu}

\newcommand{\I}[1]{I_{#1}^\alpha}
\newcommand{\Ione}[1]{I_{#1}^\alpha}

\newcommand{\J}[2]{J_{#1,#2}^\alpha}
\newcommand{\Aint}{A}
\newcommand{\Bint}{B}

\newcommand{\gen}{p}


\author{Vishwa Dewage }

\address{Department of Mathematics, Clemson University\\
South Carolina, SC 29634, USA}
\email{vdewage@clemson.edu}

\title[Toeplitz Algebra of Bounded Symmetric Domains]{Toeplitz algebra of bounded symmetric domains: A quantum harmonic analysis approach via localization}

\begin{document}


\begin{abstract}
    We prove that Toeplitz operators are norm dense in the Toeplitz algebra $\mathfrak{T}(L^\infty)$ over the weighted Bergman space $\mathcal{A}^2_\nu(\Omega)$ of a bounded symmetric domain $\Omega\subset\mathbb{C}^n$. Our methods use representation theory, quantum harmonic analysis, and weakly-localized operators. Additionally, we note that the set of all $\alpha$-weakly-localized operators form a self-adjoint algebra, containing the set of all Toeplitz operators, whose norm closure coincides with the Toeplitz algebra. We also prove an $H$-invariant version of this theorem.
\end{abstract}

\maketitle

\setcounter{tocdepth}{1}
\tableofcontents


\section{Introduction}

The density of Toeplitz operators in various operator algebras over the Bergman space $\cA^2(\Omega)$ of a domain $\Omega\subset \C^n$ is a long-standing question in operator theory. Engli\v{s} proved that Toeplitz operators are dense in the algebra of all bounded operators with respect to the strong operator topology \cite{E92}. Then Berger and Coburn showed that Toeplitz operators are norm dense in the algebras of compact operators and in trace-class operators \cite{BC94}.

The Toeplitz algebra $\toepalg$ is the $C^*$-subalgebra of bounded operators generated by Toeplitz operators with bounded symbols. As the products of Toeplitz operators are not Toeplitz operators, understanding the Toeplitz algebra is by no means an easy task.
While the density results of Engli\~s and, Berger and Coburn are independent of the domain, the density of Toeplitz operators in the Toeplitz algebra was established much later for the Bergman space over the unit ball $\cA^2(\B^n)$ and the Fock space $\cF^2(\C^n)$ by Xia \cite{X15}. In this article, we generalize this result to weighted Bergman spaces $\Berg$ of any bounded symmetric domain $\Omega$. Our techniques rely on representation theory and quantum harmonic analysis (QHA). We also remark that this means that the Toeplitz algebra is linearly generated. Further, we generalize this result to Toeplitz algebras invariant under a group action.

In his paper titled "quantum harmonic analysis on the phase space", Werner explores classical harmonic analysis notions such as convolutions and Fourier transforms for operators \cite{W84}. Eventually, QHA was found to be a powerful tool in mathematical physics, time-frequency analysis, and operator theory \cite{BBLS22, FHL24, KL21, KLSW12, LS20, S24}. As Fulsche observed in \cite{F19}, a Toeplitz operator on the Fock space can be written as a QHA convolution, making QHA a natural framework for discussing Toeplitz operators on the Fock space. Fulsche then used this framework to reprove Xia's density theorem for the Fock space. Since then, QHA has proven to be an effective tool in analyzing Toeplitz operators on Fock spaces \cite{DM23, DM24, DM24a, FH23, F19, FR23}. As for the Bergman space $\cA^2(\B^n
)$, QHA is more challenging due to the noncommutativity of the group acting on the unit ball $\B^n$.
In \cite{DDMO24}, QHA was discussed for the Bergman space $\cA^2(\B^n
)$, and is then used to discuss approximations by Toeplitz operators. In this article,  we formulate some of the QHA notions for the Bergman space $\Berg$ of a bounded symmetric domain, to be utilized in our analysis.

The other ingredient in our proof is $\alpha$-weakly localized operators, a concept introduced by Isralowitz, Mitkovski, and Wick in \cite{IMW13} for $\cA^2(\B^n)$. They prove that the $\alpha$-weakly localized operators form a self-adjoint subalgebra $\wloc$ of all bounded operators containing the set of all Toeplitz operators with bounded symbols. Then Xia proves that the norm closure $\overline{\wloc}$ coincides with the norm closure of the set of all Toeplitz operators. This in turn proves that the Toeplitz algebra coincides with $\overline{\wloc}$ and that  Toeplitz operators are norm dense in the Toeplitz algebra. Since the motivation behind weakly-localized operators is the well-known Forelli-Rudin estimates, which also hold for bounded symmetric domains, we naturally extend the definition of weakly-localized operators. We then extend Xia's result using tools from QHA.

The article is organized as follows. In Section \ref{sec:prelim}, we collect facts about bounded symmetric domains, Bergman spaces $\Berg$, and discuss a projective representation acting on $\Berg$. In Section \ref{sec:QHA}, we discuss QHA for $\Berg$: we introduce the convolution between a function and an operator and then view Toeplitz operators as QHA convolutions. Additionally, we point out that certain QHA convolutions are always contained in the Toeplitz algebra. In Section \ref{sec:wloc}, we introduce weakly localized operators for bounded symmetric domains, and show that they form a self-adjoint algebra that contains Toeplitz operators. In section \ref{sec:Toeplitz}, we establish Theorem \ref{theo:wloc_Toeplitz}, our main result, proving  Toeplitz operators are norm dense in the Toeplitz algebra $\toepalg$. We also provide several characterizations of the Toeplitz algebra. Lastly, in section \ref{sec:HToeplitz}, we prove Theorem \ref{theo:wloc_ToeplitzH}, generalizing our main results to the case of Toeplitz algebras invariant under a group action.


\section{Preliminaries}\label{sec:prelim}

\subsection{Bounded symmetric domains}
A bounded domain $\Omega\subset \C^n$ is said to be symmetric, if for each $z\in \Omega$, there exists an automorphism $\sigma_z$ of period two for which $z$ is an isolated fixed point. A comprehensive overview of the theory of bounded symmetric domains can be found in \cite{H78,H63,K69,L69,L77,S80}. We mainly follow the conventions in \cite{L77}.

The Bergman metric $g$ is a Hermitian metric on the tangent bundle $T_\Omega\C^n$, defined by
$$g_{i,j}(z)=\partial_i\Bar{\partial}_j\log K(z,z),$$
where $K(z,z)$ is the reproducing kernel of the Bergman space $\cA^2(\Omega)$ of all holomorphic functions on $\Omega$ that are square-integrable with respect to the Lebesgue measure.
We point to \cite{H78}, \cite{K01}, and \cite{L77}, as references on the Bergman metric.
 
Since $\Omega$ is bounded, $g_{i,j}(z)$ is positive definite at each $z\in \Omega$. The Bergman metric is invariant under the automorphism $\sigma_z$. Therefore, $(\Omega,g)$ is a Hermitian symmetric space. Conversely, every Hermitian symmetric space of noncompact type can be realized as a bounded symmetric domain via its Harish-Chandra realization (see Theorem 7.1 in \cite{H78}). The aforementioned domain is convex, circled, and contains the origin.

Let $\mathrm{Aut}(\Omega)$ denote the group of all automorphisms (biholomorphic maps onto itself) of $\Omega$. Let $G$ denote the connected component of $\mathrm{Aut}(\Omega)$ containing the identity. Then $G$ is a semisimple Lie group acting transitively on $\Omega$. Let $K$ denote the stabilizer
$$K=\{x\in G\mid x\cdot 0 =0\}.$$
We have that $\Omega$ is the homogeneous space $\Omega=G/K$, where the identification is made via the map $xK\mapsto x\cdot 0$. 

Given $z\in \Omega$, we define the biholomorphism $\tau_z$ by setting $\tau_z=\sigma_{z_0}$, where $z_0$ is the midpoint of the geodesic connecting $0$ and $z$. Then $\tau_z$ is an involution that interchanges $z$ and $0$.

\begin{lemma}\label{lem:Aut}
    Let $x\in G$. Then $x=\tau_{x\cdot 0}k_x$ for some $k_x\in K$.
\end{lemma}
\begin{proof}
    Note that $(\tau_{x\cdot 0}x)(0)=\act{x\cdot 0}{x\cdot 0}=0$. Hence  $\tau_{x\cdot 0}x=k_x$ for some $k_x\in K$. As $\tau_x^{-1}=\tau_x$ it follows that $x=\tau_{x\cdot 0}k_x$.
\end{proof}

A bounded symmetric domain is said to be irreducible if it is not biholomorphic to a cartesian product of bounded symmetric domains. An irreducible bounded symmetric domain $\Omega$ is characterized by three constants, the rank $r$ and multiplicities $a$ and $b$ (see \cite{E99}). The tuple $(r,a,b)$ is called the type of $\Omega$.

Throughout this article, we let $\Omega\subset \C^n$ be an irreducible bounded symmetric domain of type $(r,a,b)$ in its Harish-Chandra realization.
Let $\gen$ denote the genus $\gen=(r-1)a+b+2$. The complex dimension $n$ satisfies $n=r(r-1)a/2+rb+r$. The Jordan triple determinant $h(z,w)$ associated with $\Omega$ is a polynomial in $\overline{z}$ and $w$. This encodes many of the important properties of the associated Bergman space. For the unit ball $\B^n$ we have $(r,a,b)=(1,2,n-1)$, $\gen=n+1$ and
$$h(z,w)=1-\ip{z}{w},\quad z,w\in \B^n.$$
 The polynomial $h$ satisfies the following: for $z,w, \xi\in \Omega$,
\begin{enumerate}
    \item $h(z,w)\neq 0$
    \item $h(z,0)=1$
    \item $h(z,w)=\overline{h(w,z)}$
    \item $h(kz,kw)=h(z,w)$ for all $k\in K$
    \item $h(\act{\xi}{z},\act{\xi}{w})=h(z,w)$.
\end{enumerate}
We would often identify the Jordan triple determinant as a function on $G\times G$, which we also denote by $h$, and write $h(x,y):=h(x\cdot 0, y\cdot 0)$ for $x,y\in G$. We also have
\begin{align}\label{eq:JTdet}
    h(xy_1,xy_2)=\frac{h(x,x)h(y_1,y_2)}{h(y_1,x)h(x,y_2)}.
\end{align}

\subsection{Cartan's classification}\label{subsec:examples}
Irreducible bounded symmetric domains are classified into six types by Cartan's classification. The first four families are known as classical domains, while the last two are termed exceptional domains, as only one domain exists for each exceptional type.
\begin{align*}
I_{m_1,m_2} &: m_1 \times m_2 \text{ complex matrices } A \text{ satisfying } I - A^*A > 0. \\
II_m &: \text{ the symmetric } \text{ matrices in } I_{m,m}. \\
III_m &: \text{ the skew-symmetric } 
\text{ matrices in } I_{m,m}.\\
IV_n &: \text{The Lie ball } (z\in \mathbb{C}^n  \text{ satisfying } |\sum z_j^2|<1, \quad 1 + |\sum z_j^2|^2 - 2|z|^2 > 0). \\
V &: \text{ the unit ball of } 1 \times 2 \text{ matrices over the 8-dimensional Cayley algebra.} \\
VI &: \text{ the unit ball of } 3 \times 3 \text{ Hermitian matrices over 8-dimensional Cayley algebra.}
\end{align*}

In the following table, we list the unique constants (r,a,b) associated with each type of classical domain.\\
\begin{center}
    \begin{tabular}{|c|c|c|c|c|c|}
\hline
domain & rank $r$ & $a$ & $b$ & dimension $n$ & genus $\gen$ \\
\hline
$I_{m_1,m_2}$ ($m_1 \leq m_2$) & $m_1$ & 2 & $m_2-m_1$ & $m_1m_2$ & $m_1+m_2$ \\
$II_{m}$ & $m$ & 1 & 0 & $\frac{1}{2}m(m+1)$ & $m+1$ \\
$III_{m}$, $m=2r+\epsilon$, $\epsilon \in \{0,1\}$ & $r$ & 4 & $2\epsilon$ & $r(2r+2\epsilon-1)$ & $4r+2\epsilon-2$ \\
$IV_{n}$ & 2 & $n-2$ & 0 & $n$ & $n$ \\
\hline
\end{tabular}
\end{center}
For more information on classical domains, we refer to \cite{H63}.

\subsection{The Bergman distance}\label{subsec:distance}
Here, we discuss the Bergman distance function associated with the Bergman metric.
Note that the length of a tangent vector $\xi\in T_z\C^n$ is given by
$$|\xi|_{g,z}=\sqrt{\sum_{i,j=1}^n g_{i,j}(z)\xi_i\overline{\xi_j}}.$$
The length of a  $C^1$-curve $\gamma:[0,1]\to \Omega$ is defined as
$$length(\gamma)=\int_0^1 \Big|\odv{\gamma}{t}\Big|_{g,\gamma(t)} dt= \int_0^1 \sqrt{\sum_{i,j=1}^n g_{i,j}(\gamma(t))\gamma'_i(t)\overline{\gamma'_j(t)}} \ dt.$$

The Bergman distance $\beta:\Omega\times\Omega\to [0,\infty)$ is then given by
$$\beta(z,w)=\inf\{length(\gamma) \mid \gamma:[0,1]\to \Omega \text{ is piecewise $C^1$ s.t. $\gamma(0)=z$ and $\gamma(1)=w$}\}.$$
Then $\beta$ is a metric on $\Omega$ which is invariant under automorphisms. Since any geodesic can be extended indefinitely by repeated geodesic symmetries, $(\Omega,\beta)$ is a complete metric space by Hopf-Rinow theorem. Occasionally, the Bergman distance $\beta$ is also called the Bergman metric in the literature. However, we avoid this terminology. We define the ball with radius $\rho$ centered at $z\in \Omega$ by
$$\Dr(z)=\{w\in \Omega\mid \beta(w,z)<\rho\}.$$

The Bergman distance $\beta$ also induces an almost metric space structure on the group $G$. While the triangle inequality and symmetry are still satisfied by $\beta$ on $G\times G$, the identity of indiscernibles fails. For $x\in G$, we define the set $\Br(x)$  by
$$\Br(x)=\{y\in G\mid \beta(x\cdot 0,y\cdot 0)<\rho\}.$$

If $E\subset \C^n$ is measurable, we denote the Lebesgue measure of $E$ by $|E|$.
In the following lemma, we collect a few useful facts about these sets. 
\begin{lemma}\label{lem:balls}
    We have the following:
    \begin{enumerate}
        \item $\lim_{\rho\to \infty}|\Omega\setminus \Dr(0)|=0$.
        \item For each $\rho\in [0,\infty)$, $\Dr(0)$ is precompact.
        \item For each $\rho\in [0,\infty)$, $\Br(0)$ is precompact and has finite Haar measure.
        \item Let  $\rho\in [0,\infty)$. If $y\in \Brr{\rho/2}(x)$, then $G\setminus \Br(x)\subset G\setminus \Brr{\rho/2}(y)$
    \end{enumerate}
\end{lemma}

\begin{proof}
    Proof of (1): By the inner regularity of the Lebesgue measure there is a sufficiently large compact set $E$ that makes $|\Omega\setminus E|$ is arbitrarily small. Since $E$ is compact, we have $E\subset\Dr(0)$ for some $\rho>0$ and the result follows.
    
    Proof of (2): Recall that $(\Omega,\beta)$ is a complete metric space. Then any closed and bounded set is compact by the Hopf-Rinow theorem. Therefore, $\Dr(0)$ is precompact.
    
    Proof of (3): Note that
    $$\Br(x)=\{y\in G\mid y\cdot 0\in \Dr(x\cdot 0)\}.$$
    Then $\Br(x)\subset G$ is precompact as $\Dr(x\cdot 0)\subset G/K$ is precompact. Now since the Haar measure is a Radon measure, we have $\mu_G(\Br(x))<\infty$.
    
    Proof of (4): By triangle inequality, $\Brr{\rho/2}(y)\subset \Br(x)$. Therefore,
    $G\setminus\Br(x)\subset G\setminus \Brr{\rho/2}(y)$.
\end{proof}

\subsection{The Bergman space}\label{subsec:Berg}
Fix $\nu>-1$ and define the measure
$$\dvz{z}=c_\nu h(z,z)^{\nu} dz$$
where $dz$ denotes the normalized Lebesgue measure on $\Omega$, and $c_\nu>0$ is such that $\int_\Omega \dv=1$.
\begin{remark}
    Assuming that $\nu>-1$ guarantees that $\dv$ is a probability measure. However, as we discuss $\alpha$-weakly localized operators, we need an extra restriction $\nu>-1+(r-1)a/2$ for our theory to work.
\end{remark}
The Bergman space $\Berg$ consists of all holomorphic functions that are square-integrable with respect to $\dv$. The inner-product and norm in $\Berg$ are denoted by $\ipnu{\cdot}{\cdot}$ and $\normnu{\cdot}$ respectively. The Bergman space $\Berg$ is a reproducing kernel Hilbert space with reproducing kernel $K^\nu$ given by
\[
   K^\nu(w,z)=K_z^\nu(w)=h(z,w)^{-(\nu+\gen)},\ \ z,w\in \Omega.
\]
We have the reproducing formula
$$f(z)=\ipnu{f}{K^\nu_z},\quad f\in \Berg, z\in \Omega.$$
The normalized reproducing kernels $k_z^\nu$ are given by
$$k_z^\nu(w)=\frac{K_z^\nu(w)}{\normnu{K_z^\nu}}=\frac{h(z,z)^{\frac{1}{2}(\nu+\gen)}}{h(z,w)^{\nu+\gen}},\ \ z,w\in \Omega.$$
The Bergman space $\Berg$ is a closed subspace of the space of all square-integrable functions $L^2(\Omega,\dv)$. The Bergman projection $P:L^2(\Omega,\dv)\to \Berg$ is given by
\[
    (Pf)(z)=\ipnu{f}{K_z^\nu},\ \ z\in \Omega,\ \ f\in L^2(\Omega,\dv).
\]

The algebras of bounded operators and trace-class operators on $\Berg$ are denoted by by $\bdd$ and $\traceop$ respectively.
Given a function $a:\C^n\to \C$ the Toeplitz operator $T_a$ with symbol $a$ is defined formally by 
\[
   T_af=P(af),\ \ f\in \Berg.
\]
If $a\in \bddf$, then $T_a\in \bdd$ with $\|T_a\|\leq \|a\|_\infty$.
The \textit{Toeplitz algebra} $\toepalg$ 
is the $C^*$-subalgebra of $\bdd$  generated by the Toeplitz operators with symbols in $\bddf$.
Toeplitz operators on Bergman spaces of bounded symmetric domains has been a popular topic over the years \cite{E99, DOQ15, DOQ18, H17, H19, U83}. The rich geometry of $\Omega$ provides a valuable set of tools to study them. 

The Berezin transform is an important notion in analyzing Toeplitz operators. The Berezin transform of a bounded operator $S\in \bdd$ is given by 
$$B(S)(z)=\ipnu{Sk_z^{\nu}}{k_z^{\nu}},\quad z\in \Omega.$$
Then $\|B(S)\|_\infty\leq \|S\|$.
It is well-known that the Berezin transform $B:\bdd\to \bddf$ is an injective map (the proof is similar to the proof of Proposition 3.1 in \cite{Z12}). Oftentimes, the Berezin transform can be viewed as the Banach space adjoint of the Toeplitz map $a\mapsto T_a$, by defining the latter maps on suitable Banach spaces of operators and functions.
The Berezin transform of a function $a$ is defined formally by
$$B(a)(z)=\ipnu{ak_z^{\nu}}{k_z^{\nu}}, \quad z\in \Omega.$$
If $a\in \bddf$, then $\|B(a)\|_\infty\leq \|a\|_\infty$.
We have that the Berezin transform of a Toeplitz operator is the same as the Berezin transform of the symbol, i.e. $B(T_a)=B(a)$, and that $B:\bddf\to \bddf$, $a\mapsto B(a)$ is injective. As a consequence, it follows that the map $\bddf\to \bdd,\ a\mapsto T_a$ is injective.

\begin{remark}
    There is much more to be said about the connections between the Berezin transform, the Toeplitz map, and quantum harmonic analysis. However, not to derail from the main point of this article, we limit our discussion with the hope of providing an in-depth discussion in an upcoming article co-authored with Matthew Dawson and Gestur Olafsson.
\end{remark}

\subsection{Functions on $\Omega$ as functions on the group}
Here we recall a few facts about integration on homogeneous spaces. We point to \cite[Chapter 6]{FO14} for a detailed discussion.
Let $\mu_G$ denote the Haar measure on $G$. Let $\mu_K$ denote the normalized Haar measure on $K$. 
Since $K$ is compact, there is a $G$-invariant measure $\mu_{G/K}$ on $G/K$.
For $f\in \LoneG$, we define the projection of $f$ onto $L^1(G/K)$ by
$$f_K(xK)=\int_K f(xk) \ d\mu_K(k).$$
Then 
$$\int_G f(x) \haar{x}=\int_{G/K} f_K(xK) \ d\mu_{G/K}(xK), \quad f\in \LoneG.$$
Since $\haar{K}<\infty$, we have that $f_K(xK)=f(x)$ for a right $K$-invariant function $f$.

Recall that we make the identification $G/K=\Omega$ via the map $xK\mapsto x\cdot 0$. We define a $G$-invariant measure on $\Omega$ by
$$d\lambda(z)=\frac{1}{h(z,z)^{\gen}}dz,$$
where $p$ is the genus of $\Omega$.
A function on the domain $f:\Omega\to \C$ can be identified with a function on the group, $\Tilde{f}:\G\to \C$ by setting
$$\Tilde{f}(x)=f(x\cdot 0),\quad x\in G.$$
Since all $G$-invariant measures on a homogeneous space are equivalent, by choosing a suitable normalization of $\mu_G$ we get that for $f\in \Lone$,
\begin{equation}\label{eq:integral}
    \int_{G} \Tilde{f}(x) \haar{x}=\int_{\Omega} f(z) \inv{z}.
\end{equation}

Then, we have 
\begin{equation}\label{eq:integral_nu}
    c_\nu\int_{G} \Tilde{f}(x) h(x,x)^{\nu+\gen} \haar{x}=\int_{\Omega} f(z) \dvz{z}.
\end{equation}

\subsection{A representation of $G$ on the Bergman space}

As noted in \cite{FK90}, the relatively discrete series representation of the universal covering group $\tilde{G}$ of $G$ acts on $\Berg$. However, we instead consider a projective representation for simplicity. Another approach to this can be found in \cite{CO21}.

We define the translation of a function $f:\Omega \to\C$ by $x\in G$ is given by
$$(\actfL{x}{f})(w)=f(x^{-1}\cdot w)$$
For $x\in G$, we define the unitary operators $\pinu(x)$ on $\Berg$ by
$$\pinu(x)f=k_{x\cdot 0}^\nu\actfL{x}{f}, \ \ f\in \Berg.$$
The mapping $\pinu$ is an irreducible projective unitary representation of $G$ acting on $\Berg$. For $x,y\in G$,  $$\pinu(x)\pinu(y)=m_\nu(x,y)\pinu(xy)$$
    where $m_\nu(x,y)=\frac{|h(x,y)|^{\nu+\gen}}{h(x,y)^{\nu+\gen}}\in \T$. 
Moreover, the adjoint of $\pinu(x)$ is given by $$\pinu(x)^*=\frac{1}{m(x,x^{-1})}\pinu(x^{-1}).$$

\begin{remark}
    We also point out that $\pinu(x)$ differs from the self-adjoint operators $U_z$ that are often used in the literature on Bergman spaces \cite{H19, E99}. These operators are given by
$$(U_zf)(w)=k_z(w)f(\act{z}{w}),\quad z,w\in \Omega, f\in \Berg.$$
\end{remark}

\subsection{Schur orthogonality relations}\label{subsec:Schur}
Given $f_1,f_2\in \Berg$, the matrix coefficients $\pi^\nu_{f_1,f_2}:G\to\C$ are given by
$$\pi^\nu_{f_1,f_2}(g)=\ipnu{f_1}{\pinu(g)f_2},\quad g\in G.$$
A representation is said to be square-integrable if the matrix coefficient function is square-integrable on the group for some nonzero vectors $f_1$ and $f_2$. It is easy to verify that $\pinu$ is square-integrable, for example, by considering $f_1=f_2=1$. Square-integrable irreducible representations of locally compact unimodular groups satisfy Schur-orthogonality relations (see \cite[Proposition 14.3.3]{D83}), which we state below for $\pi_\nu$. We have that:
\begin{equation}\label{eq:Schur}
    \ip{\pi^\nu_{f_1,f_2}}{\pi^\nu_{f_3,f_4}}_{G}=\frac{1}{d_{\pinu}}\ipnu{f_1}{f_3}\ipnu{f_2}{f_4},\quad f_i\in \Berg,
\end{equation}
where $\langle\cdot, \cdot\rangle_G$ denotes the inner product of $L^2(G,\mu_G)$, and $d_{\pinu}$ is  a constant called the formal dimension of $\pi$. While in \cite{D83} and \cite{S83}, the orthogonality relations are proved for irreducible square-integrable representations instead of an irreducible square-integrable projective representations such as $\pi_\nu$, the proof in \cite{S83} still works to prove this widely used identity.

In the following lemma, we compute the formal dimension $d_{\pinu}$ of $\pi_\nu$.

\begin{lemma}
    The formal dimension $d_{\pinu}$ is given by $d_{\pinu}=c_\nu$.
\end{lemma}

\begin{proof}
    Note that by Schur orthogonality relations 
    \begin{align*}
        \ip{\pi^\nu_{1,1}}{\pi^\nu_{1,1}}_G&=\frac{1}{d_{\pinu}}\ipnu{1}{1}\ipnu{1}{1}\\
        &= \frac{1}{d_{\pinu}}.
    \end{align*}
    Also,
    \begin{align*}
        \ip{\pi^\nu_{1,1}}{\pi^\nu_{1,1}}_G&=\int_G \ipnu{1}{\pinu(x)1}\ipnu{\pinu(x)1}{1}\haar{x}\\
        &=\int_G h(x,x)^{\nu+\gen} \ipnu{1}{K_{x\cdot 0}^\nu}\ipnu{K_{x\cdot 0}^\nu}{1}\haar{x} \quad \text{(as $\pinu(x)1=k^\nu_{x\cdot 0}$)} \\
        &=\frac{1}{c_\nu}\int_{\B^n}\ipnu{1}{K_{z}^\nu}\ipnu{K_{z}^\nu}{1} \dvz{z}\ \ \text{(by \ref{eq:integral_nu})}\\ 
        &=\frac{1}{c_\nu}\int_{\B^n} 1 \ \dvz{z} \quad \text{(by the reproducing formula)}\\
        &=\frac{1}{c_\nu},
    \end{align*}
    completing the proof.
\end{proof}


\section{Quantum harmonic analysis on the Bergman space} \label{sec:QHA}

 In \cite{DDMO24}, QHA for the Bergman space over the unit ball is discussed.  Here, we collect some of the QHA notions for bounded symmetric domains. A comprehensive discussion of QHA on Bergman spaces of bounded symmetric domains will be addressed in future work. Now, we proceed to discuss convolutions between functions, convolutions between functions and operators, and then we express Toeplitz operators as convolutions.

Given functions $\psi:G\to \C$ and $a:\Omega\to \C$, we define the convolution $\psi\ast a:\Omega\to \C$ formally by
$$(\psi\ast a)(z)=\int_{G} \psi(x)(\actfL{x}{a})(z)\haar{x}=\int_{G} \psi(x)a(x^{-1}\cdot z)\haar{x}.$$
Let $q,q'\in[1,\infty]$ be s.t. $\frac{1}{q}+\frac{1}{q'}=1$. Then by the Young's inequality we have the following:
If $\psi\in L^q(G)$ and $a\in L^{q'}(\Omega,d\lambda)$ then $\psi\ast a\in \bddf$ with\\ 
    $$\|\psi\ast a\|_\infty\leq \|\psi\|_{q}\|a\|_{q'}.$$

The translation of an operator $S\in \bdd$ by $x\in G$ is given by
$$\actopL{x}{S}=\pinu(x)S\pinu(x)^*.$$

The map $x\mapsto L_x^\nu$ is a representation of $G$ acting on the Banach space $\bdd$, and satisfies
\begin{align}\label{eq:homomorph}
    L_{xy}^\nu=L_x^\nu\circ L_y^\nu, \quad x,y \in G.
\end{align}

Now we define the convolution of a function and an operator. 
Given a function $\psi:G\to \C$ and and an operator $S\in \bdd$, we define $\psi\ast S$ formally by the weak-integral
$$\psi\ast S=\int_G\psi(x) \actopL{x}{S} \haar{x}.$$
Since $G$ is a unimodular group and $\pi_\nu$ is a square-integrable irreducible representation of $G$, the following QHA-Young's inequalities hold:
\begin{enumerate}
    \item If $\psi\in \LoneG$ and $S\in \bdd$, then $\psi\ast S\in \bdd$ with\\
    $\|\psi\ast S\|\leq \|\psi\|_1\|S\|$.
    \item If $\psi\in \bddfG$ and $S\in \traceop$, then $\psi\ast S\in \bdd$ with\\
    $\|\psi\ast S\|\leq \frac{1}{d_{\pinu}}\|\psi\|_\infty\|S\|_\tr$.
\end{enumerate}
The first inequality above is not hard to verify. Since $G$ is a unimodular group and $\pi_\nu$ is a square-integrable irreducible projective representation of $G$, the Schur orthogonality relations hold (see subsection \ref{subsec:Schur}).
The second inequality is a consequence of this. For an explicit proof of this inequality, we point to \cite[Proposition A.7]{DDMO24}.

\begin{remark}
    The inequality (2) above was initially proved by Halvdansson in \cite[Lemma 4.12]{H23}. There, the author extends Werner's QHA, which was originally constructed for the abelian setting of the phase space, to locally compact groups that are not necessarily unimodular. However, when a unimodular group (as is the case with $G$) is considered, the proof is less complicated as seen in \cite[Proposition A.7]{DDMO24}. Moreover, in \cite{H23} and \cite{DDMO24}, a representation is considered instead of a projective representation. Nevertheless, the proofs remain valid, as the Schur orthogonality relations are still valid for irreducible square-integrable projective representations.
\end{remark}

Now due to \ref{eq:homomorph}, the above convolutions are associative. Assume that  $\psi_1,\psi_2:G\to \C$ and $S\in \bdd$ are functions and an operator s.t. the convolutions given below are well defined by the QHA Young's inequalities. Then
$$\psi_1\ast (\psi_2\ast S)=(\psi_1\ast \psi_2)\ast S.$$

\subsection{Toeplitz operators}
Here, we prove that Toeplitz operators are QHA convolutions. Let $\Phi$ be the rank one operator $\Phi=1\otimes 1$. 

\begin{lemma}\label{lem:Toeplitz}
    Let $1\leq p\leq \infty$ and let $a\in \Lp$. Then the Toeplitz operator $T_a$ can be written as
    $$T_a=c_\nu (a\ast \Phi)= a\ast (c_\nu \Phi).$$
\end{lemma}

\begin{proof}
    First note that 
    $$\actopL{x}{\Phi}=\pinu(x)1\otimes \pinu(x)1= h(x,x)^{\nu+\gen} (K_{x\cdot 0}^\nu\otimes K_{x\cdot 0}^\nu).$$
    Note that we consider $a$ as a function on $G$ as usual, by writing $a(x)=a(x\cdot 0)$, for $x\in G$. Then for $f_1,f_2\in \Berg$,
    \begin{align*}
        \ipnu{(a\ast \Phi)f_1}{f_2} &= \int_G a(x\cdot 0) \ipnu{\actopL{x}{\Phi}f_1}{f_2} \haar{x}\\
        &= \int_G h(x,x)^{\nu+\gen} a(x\cdot 0) \ipnu{(K_{x\cdot 0}^\nu\otimes K_{x\cdot 0}^\nu)f_1}{f_2} \haar{x}\\
        &= \frac{1}{c_\nu} \int_{\Omega} a(z) \ipnu{(K_z^\nu\otimes K_z^\nu)f_1}{f_2} \dvz{z}\ \ \text{(by \ref{eq:integral_nu})}\\
        &=\frac{1}{c_\nu}\int_{\Omega} a(z) f_1(z)\overline{f_2(z)} \dvz{z}\\
        &=\frac{1}{c_\nu} \ipnu{af_1}{f_2}\\
        &=\frac{1}{c_\nu} \ipnu{T_af_1}{f_2}.\qedhere
    \end{align*}
\end{proof}

It is well-known that Toeplitz operators with symbols in $\Lone$ are in $\traceop$. This fact also follows from the above lemma and the QHA Young's inequality.
On the other hand, Berger and Coburn proved that the set of all Toeplitz operators with continuous compactly supported symbols forms a dense subset of $\traceop$ \cite[Theorem 9]{BC94}. While this Theorem was explicitly stated for the case of the Fock space, the authors note, as they begin section 3 in \cite{BC94}, that this holds for any Bergman space over an arbitrary bounded domain. Indeed, the main ingredients of their proof are the injectivity of the Berezin transform and standard duality arguments, which do not depend on the nature of the domain. Additionally, we note that such techniques are also commonplace in the literature on QHA (see, for example, Theorem 5.2 in \cite{DDMO24}). Furthermore, as compactly supported continuous functions are contained in $\Lone$, we have that 
Toeplitz operators with symbols in $\Lone$ form a dense subset of $\traceop$ \cite[Theorem 9]{BC94}. We use this result in the following lemma.

\begin{lemma}\label{lem:convol_traceclass}
    We have the following inclusion
    $$\bddfG\ast \traceop\subset \overline{\{T_a\in \bdd \mid a\in \bddf\}}.$$
\end{lemma}

\begin{proof}
    Let $\psi \in \bddfG$, $S\in \traceop$ and let $\epsilon>0$. Since Toeplitz operators whose symbols are in $\Lone$ forms a dense subset of $\traceop$, there is $a\in \Lone$ s.t. 
    $$\|T_a-S\|_\tr<\frac{c_\nu\epsilon }{\|\psi\|_\infty+1}.$$
    Then
    $$\|\psi\ast T_a-\psi\ast S\|=\|\psi\ast (T_a-S)\|\leq \frac{1}{c_\nu}\|\psi\|_\infty\|T_a-S\|_\tr<\epsilon.$$
    Also,
    $$\psi\ast T_a=\psi\ast(a\ast c_\nu\Phi)=(\psi\ast a)\ast c_\nu\Phi=T_{\psi\ast a}.$$
    Finally, we note that $\psi\ast a \in \bddf$. This completes the proof. 
\end{proof}


\section{Weakly localized operators}\label{sec:wloc}

In \cite{IMW13}, the authors introduce the notion of weakly localized operators for both the Bergman space over the unit ball and the Fock space.
In this section, we generalize this definition for Bergman spaces over bounded symmetric domains. Then we prove that weakly localized operators form a self-adjoint algebra containing Toeplitz operators. Hitherto, we have only assumed that $\nu>-1$. Throughout this section, we make the restriction $\nu>-1+(r-1)a/2$.

\subsection{Forelli-Rudin estimates}
   Now we recall the Forelli-Rudin estimates from \cite[Theorem 4.1]{FK90}. Let $t>-1$ and $c\in\R$. Then
   \begin{enumerate}
       \item If $c-(t+\gen)<-(r-1)a/2$, we have that 
   \begin{align}\label{eq:FRone}
       \int_{\Omega} \frac{h(w,w)^t}{h(z,w)^c} dw \leq C_{t,c},\quad z\in \Omega
   \end{align}
   for some $C_{t,c}>0$.
   \item If $c-(t+\gen)>(r-1)a/2$, we have that 
   \begin{align}\label{eq:FRtwo}
       \int_{\Omega} \frac{h(w,w)^t}{h(z,w)^c} dw \leq \frac{C_{t,c}}{h(z,z)^{c-(t+\gen)}},\quad z\in \Omega
   \end{align}
   for some $C_{t,c}>0$.
   \end{enumerate}

\subsection{Weakly localized operators}
The motivation for the following definition of weakly-localized operators are the Forelli-Rudin estimates. Recall the set $\Br(x)$ defined in Subsection \ref{subsec:distance}.

\begin{definition}\label{def:weakly}
    For $S\in \bdd$, $\alpha\in (\frac{\gen-\nu-2}{2},\frac{\nu+b+2}{2})$,  and for $\rho\geq 0$, we define 
    \begin{align}
        \Ione{S}(\rho)=\sup_{x\in G}\int_{G\setminus \Br(x)} |\ipnu{S\pinu(x)1}{\pinu(y)1}|\frac{h(y,y)^{\alpha}}{h(x,x)^{\alpha}} \haar{y}
    \end{align}
    An operator $S$ is $\alpha$-localized if the following conditions are satisfied:
    \begin{enumerate}
        \item $\Ione{S}(0)<\infty$
        \item $\Ione{S^*}(0)<\infty$
        \item $\lim_{\rho\to \infty}\Ione{S}(\rho)=0$
        \item $\lim_{\rho\to \infty}\Ione{S^*}(\rho)=0$.
    \end{enumerate}
\end{definition}
We denote the set of all $\alpha$-localized operators by $\wloc$. 

\begin{remark}
    Note that the interval $(\frac{\gen-\nu-2}{2},\frac{\nu+b+2}{2})$ mentioned above is nonempty as we assumed that $\nu>-1+(r-1)a/2$.
\end{remark}

First, we use Forelli-Rudin estimates to show that the identity operator is weakly localized. The proof of this fact follows similar to the proof of \cite[Lemma 10]{E99}.

\begin{lemma}\label{lem:identity_wloc}
    Let $\alpha\in (\frac{\gen-\nu-2}{2},\frac{\nu+b+2}{2})$ and let $\id\in \bdd$ be the identity operator. Then $\id\in \wloc$.
\end{lemma}
\begin{proof}
    First note that
    \begin{align*}
        \Ione{\id}(\rho)&= \sup_{x\in G}\int_{G\setminus \Br(x)} |\ipnu{\pinu(x)1}{\pinu(y)1}|\frac{h(y,y)^{\alpha}}{h(x,x)^{\alpha}} \haar{y}\\
        &=\sup_{x\in G}\int_{G\setminus \Br(0)} |\ipnu{\pinu(x)1}{\pinu(xy)1}|\frac{h(xy,xy)^{\alpha}}{h(x,x)^{\alpha}} \haar{y},
    \end{align*}
    where the last line follows by a change of variable.
    Also,
    $$\ipnu{\pinu(x)1}{\pinu(xy)1}=\ipnu{1}{\pinu(y)1}=h(y,y)^{(\nu+\gen)/2},$$
    and due to \ref{eq:JTdet}, we have
    $$\frac{h(xy,xy)^{\alpha}}{h(x,x)^{\alpha}}=\frac{h(y,y)^{\alpha}}{|h(x,y)|^{2\alpha}}.$$
    Hence
    \begin{align*}
        \Ione{\id}(\rho)
        &=\sup_{x\in G}\int_{G\setminus \Br(0)} h(y,y)^{(\nu+\gen)/2}\frac{h(y,y)^{\alpha}}{|h(x,y)|^{2\alpha}} \haar{y}\\
        &=\sup_{z\in \Omega}\int_{\Omega\setminus \Dr(0)} \frac{h(w,w)^{\alpha+(\nu-\gen)/2}}{|h(z,w)|^{2\alpha}} \ dw\ \ \text{(by \ref{eq:integral})}. \\
    \end{align*}
    
    Let $q\in \big(1,\frac{\gen-(r-1)a/2}{\alpha+(\gen-\nu)/2}\big)$. To verify that the interval is nonempty, note that $$\alpha<(\nu+b+2)/2=(\gen+\nu-(r-1)a)/2.$$ This implies that
    $$\frac{\gen-(r-1)a/2}{\alpha+(\gen-\nu)/2}>1.$$
    Let $q'>1$ s.t. $\frac{1}{q}+\frac{1}{q'}=1$. 
    
    Then by Holder's inequality,
    \begin{align*}
         \Ione{\id}(\rho)&\leq   |\Omega\setminus \Dr(0)|^\frac{1}{q'} \Big(\int_{\Omega} \frac{h(w,w)^{q(\alpha+(\nu-\gen)/2)}}{|h(z,w)|^{2q\alpha}} \ dw\Big)^\frac{1}{q}\\
            &\leq |\Omega\setminus \Dr(0)|^\frac{1}{q'} C_{t,c}^{\frac{1}{q}},
    \end{align*}
    where the last line follows by applying Forelli-Rudin estimate \ref{eq:FRone} with $t=q(\alpha+(\nu-\gen)/2)$ and $c=2q\alpha$. Also, we verify
    \begin{align*}
        t&=q(\alpha+(\nu-\gen)/2)\\
        &\geq \alpha+(\nu-\gen)/2\\
        &> (\gen-\nu-2)/2+(\nu-\gen)/2\\
        &=-1,
    \end{align*}
    and
        \begin{align*}
            c-(\gen+t)&=(\alpha+(\gen-\nu)/2)q-\gen\\
            &<(\alpha+(\gen-\nu)/2)\frac{\gen-(r-1)a/2}{\alpha+(\gen-\nu)/2}-\gen \\
            &=-(r-1)a/2,
        \end{align*}
        ensuring the validity of the application of the Forelli-Rudin estimate \ref{eq:FRone}. Since $|\Omega\setminus \Dr(0)|\to 0$, as $\rho\to \infty$, it follows that
        $$\lim_{\rho \to \infty}\Ione{\id}(\rho)=0.$$
        Hence $\id\in \wloc$.
\end{proof}

\subsection{The algebra of weakly localized operators}

Here we prove that $\wloc$ is a self-adjoint subalgebra of $\bdd$ that contains the set of all Toeplitz operators. However, $\wloc$ may not be closed. But the closure $\overline{\wloc}$ is a $C^*$-algebra. Our proof relies on the following lemma.

For $S,T\in \bdd$, $\alpha\in (\frac{\gen-\nu-2}{2},\frac{\nu+b+2}{2})$ and $\rho\geq 0$, we define the integral $\J{S}{T}(\rho)$ by
       \begin{align*}
           \J{S}{T}(\rho)=\sup_{x\in G}\int_{G} &|\ipnu{S\pinu(x)1}{\pinu(g)1}| \\
           &\times \int_{G\setminus \Br(x)}  |\ipnu{\pinu(g)1}{T\pinu(y)1}| \frac{h(y,y)^{\alpha}}{h(x,x)^{\alpha}} \haar{y} \haar{g} .
       \end{align*}

    \begin{lemma}\label{lem:Jbeta}
        Let $S,T\in \bdd$, $\alpha\in (\frac{\gen-\nu-2}{2},\frac{\nu+b+2}{2})$ and $\rho\geq 0$.
        Then 
        $$\J{S}{T}(\rho)\leq \I{T}(\rho/2) \I{S}(0)+\I{T}(0)\I{S}(\rho/2).$$
    \end{lemma}

    \begin{proof}
        First, note that
        \begin{align*}
            \J{S}{T} (\rho)
            &= \sup_{x\in G} \  (\Aint(x,\rho)+\Bint(x,\rho))
        \end{align*}
        where
        \begin{align*}
            \Aint(x,\rho)=\int_{\Brr{\rho/2}(x)} |\ipnu{S\pinu(x)1}{\pinu(g)1}| \int_{G\setminus \Br(x)}  |\ipnu{\pinu(g)1}{T\pinu(y)1}| \frac{h(y,y)^{\alpha}}{h(x,x)^{\alpha}} \haar{y} \haar{g}
        \end{align*}
        and 
        \begin{align*}
            \Bint(x,\rho)=\int_{G\setminus \Brr{\rho/2}(x)} |\ipnu{S\pinu(x)1}{\pinu(g)1}| \int_{G\setminus \Br(x)}  |\ipnu{\pinu(g)1}{T\pinu(y)1}| \frac{h(y,y)^{\alpha}}{h(x,x)^{\alpha}} \haar{y} \haar{g}.
        \end{align*}
        Note that for $g\in \Brr{\rho/2}(x)$, we have $G\setminus \Br(x)\subset G\setminus \Brr{\rho/2}(g)$ by Lemma \ref{lem:balls}. 
        Therefore,
        \begin{align*}
            \Aint(x,\rho)&\leq \int_{\Brr{\rho/2}(x)} |\ipnu{S\pinu(x)1}{\pinu(g)1}|  \int_{G\setminus \Brr{\rho/2}(g)}  |\ipnu{\pinu(g)1}{T\pinu(y)1}| \frac{h(y,y)^{\alpha}}{h(x,x)^{\alpha}} \haar{y} \haar{g}\\
            &=\int_{\Brr{\rho/2}(x)} |\ipnu{S\pinu(x)1}{\pinu(g)1}| \frac{h(g,g)^{\alpha}}{h(x,x)^{\alpha}}\\
            & \hspace{4cm} \times
            \int_{G\setminus \Brr{\rho/2}(g)}  |\ipnu{\pinu(g)1}{T\pinu(y)1}| \frac{h(y,y)^{\alpha}}{h(g,g)^{\alpha}} \haar{y} \haar{g}\\
            &< \I{T}(\rho/2) \I{S}(0).
        \end{align*}
        Also,
    \begin{align*}
        \Bint(x,\rho)&\leq \int_{G\setminus\Brr{\rho/2}(x)} |\ipnu{S\pinu(x)1}{\pinu(g)1}| \int_{G}  |\ipnu{\pinu(g)1}{T\pinu(y)1}| \frac{h(y,y)^{\alpha}}{h(x,x)^{\alpha}} \haar{y} \haar{g}\\
        &=\int_{G\setminus\Brr{\rho/2}(x)} |\ipnu{S\pinu(x)1}{\pinu(g)1}| \frac{h(g,g)^{\alpha}}{h(x,x)^{\alpha}}\\
        &\hspace{4cm} \times
            \int_{G}  |\ipnu{\pinu(g)1}{T\pinu(y)1}| \frac{h(y,y)^{\alpha}}{h(g,g)^{\alpha}} \haar{y} \haar{g}\\
        &\leq \I{T}(0)\I{S}(\rho/2).
    \end{align*}
    Therefore,
    $\J{S}{T}(\rho)\leq \I{T}(\rho/2) \I{S}(0)+\I{T}(0)\I{S}(\rho/2).$
    \end{proof}

\begin{proposition}
    For each $\alpha\in (\frac{\gen-\nu-2}{2},\frac{\nu+b+2}{2})$, we have that  $\wloc$ is a self-adjoint subalgebra of $\bdd$.
\end{proposition}
\begin{proof}
    It is clear that $\wloc$ is a self-adjoint linear subspace of $\bdd$. To prove $\wloc$ is an algebra, let $S,T\in \wloc$. Then for $\rho\geq 0$,
    \begin{align*}
        \Ione{ST}(\rho)&=\sup_{x\in G}\int_{G\setminus \Br(x)} |\ipnu{T\pinu(x)1}{S^*\pinu(y)1}|\frac{h(y,y)^{\alpha}}{h(x,x)^{\alpha}} \haar{y}\\
        &=\sup_{x\in G}\int_{G\setminus \Br(x)} \Big| \int_G  \ipnu{T\pinu(x)1}{\pinu(g)1}\ipnu{\pinu(g)1}{S^*\pinu(y)1}  \haar{g}\Big|\\
        &\hspace{8cm} \times \frac{h(y,y)^{\alpha}}{h(x,x)^{\alpha}} \haar{y} \\
        &\hspace{4cm}\text{(by \ref{eq:integral_nu} and the reproducing formula)}\\
        &\leq\sup_{x\in G}\int_{G} |\ipnu{T\pinu(x)1}{\pinu(g)1}|  \\
        &\hspace{2cm} \times \int_{G\setminus \Br(x)}  |\ipnu{\pinu(g)1}{S^*\pinu(y)1}| \frac{h(y,y)^{\alpha}}{h(x,x)^{\alpha}} \haar{y} \haar{g} \\
        &\hspace{4cm}\text{(by Fubini's theorem)}\\
        &=\J{T}{S^*}(\rho)\\
        &\leq \Ione{T}(\rho/2) \Ione{S^*}(0)+\Ione{T}(0)\Ione{S^*}(\rho/2) \quad \text{(by Lemma \ref{lem:Jbeta})}.
    \end{align*}
    Then, since $S,T\in \wloc$,
    $\lim_{\rho\to \infty} \Ione{ST}(\rho)=0.$
    Therefore, $ST\in \wloc$.
\end{proof}

\subsection{Toeplitz operators are weakly localized}
In this subsection, we prove that Toeplitz operators with $L^\infty$-symbols are weakly localized.
\begin{proposition}\label{prop:Toeplitz_wloc}
        Let $\alpha\in (\frac{\gen-\nu-2}{2},\frac{\nu+b+2}{2})$ and $a\in \bddf$. Then $T_a\in\wloc$.
    \end{proposition}
    \begin{proof}
        We have that $(T_a)^*=T_{\Bar{a}}$. Therefore, it is enough to verify only the statements (1) and (3) in Definition \ref{def:weakly}. Note that
        \begin{align*}
            \Ione{T_a}(\rho)&=\sup_{x\in G}\int_{G\setminus \Br(x)} |\ipnu{T_a\pinu(x)1}{\pinu(y)1}|\frac{h(y,y)^{\alpha}}{h(x,x)^{\alpha}} \haar{y}.
        \end{align*}
        Since $T_a=c_\nu(a\ast\Phi)$, by Lemma \ref{lem:Toeplitz}, we have
        \begin{align*}
            \ipnu{T_a\pinu(x)1}{\pinu(y)1} &=c_\nu\ipnu{(a\ast\Phi)\pinu(x)1}{\pinu(y)1}\\
            &=c_\nu\int_G a(g\cdot 0) \ipnu{\actopL{g}{\Phi}\pinu(x)1}{\pinu(y)1}\haar{g}\\
            &=c_\nu\int_G a(g\cdot 0) \ipnu{(\pinu(g)1\otimes\pinu(g)1)\pinu(x)1}{\pinu(y)1}\haar{g}\\
            &=c_\nu\int_G a(g\cdot 0) \ipnu{\pinu(x)1}{\pinu(g)1}\ipnu{\pinu(g)1}{\pinu(x)1}\haar{g}.
        \end{align*}
        Hence,
        \begin{align*}
            \Ione{T_a}(\rho)&\leq c_\nu \|a\|_\infty \sup_{x\in G}\int_{G\setminus \Br(x)} \\
            &\hspace{.5cm} \times\int_G |\ipnu{\pinu(x)1}{\pinu(g)1}\ipnu{\pinu(g)1}{\pinu(x)1}|\haar{g}\frac{h(y,y)^{\alpha}}{h(x,x)^{\alpha}} \haar{y}\\
            &=c_\nu \|a\|_\infty\J{\id}{\id}(\rho) \hspace{3cm}  \text{(by Tonelli's theorem)}\\
            &\leq 2 c_\nu \|a\|_\infty \Ione{\id}(\rho/2) \Ione{\id}(0) \hspace{2cm}  \text{(by Lemma \ref{lem:Jbeta})}.
        \end{align*}
        This completes the proof, as $\id$ is weakly localized by Lemma \ref{lem:identity_wloc}.  
    \end{proof}


\section{The Toeplitz algebra}\label{sec:Toeplitz}
In \cite{X15}, Xia proves that Toeplitz operators are norm dense in the Toeplitz algebra of the unweighted Bergman space $\cA^2(\B^n)$ over the unit ball. For this, he uses weakly localized operators.
In this section, our main goal is to generalize this theorem for the weighted Bergman space $\Berg$ over a bounded symmetric domain. Let
$$\toep:=\{T_a\in \bdd \mid a\in \bddf\}.$$
Since $\overline{\wloc}$ is a $C^*$-algebra containing the set of all Toeplitz operators with $L^\infty$-symbols, it is sufficient to prove
$$\wloc\subset \overline{\toep}.$$
While we use techniques from quantum harmonic analysis that are quite natural to bounded symmetric domains, the main ideas at play are not entirely different from Xia's proof for the unit ball. 

The following integral representation of a bounded operator is what motivates the rest of the proof.
\begin{lemma}\label{lem:integral_S}
    Let $S\in \bdd$. Then
    $$S= c_\nu^2\int_G\int_G \ipnu{S\pinu(y)1}{\pinu(x)1}\  (\pinu(x)1\otimes \pinu(y)1) \haar{x}\haar{y}.$$
\end{lemma}

\begin{proof}
    Since the identity operator can be written as the Toeplitz operator whose symbol is the constant function one, we have that
    $$S=\id S \id= T_1 S T_1=c_\nu^2(1\ast \Phi)S(1\ast \Phi).$$
    Therefore,
    \begin{align*}
        S&=c_\nu^2\int_G (1\ast \Phi)S\actopL{y}{\Phi} \haar{y}\\
        &= c_\nu^2\int_G \int_G \actopL{x}{\Phi}S\actopL{y}{\Phi} \haar{x}\haar{y}\\
        &= c_\nu^2\int_G \int_G (\pinu(x)1\otimes\pinu(x)1)S(\pinu(y)1\otimes\pinu(y)1) \haar{x}\haar{y}\\
        &= c_\nu^2\int_G\int_G \ipnu{S\pinu(y)1}{\pinu(x)1}\  (\pinu(x)1\otimes \pinu(y)1) \haar{x}\haar{y}.\qedhere
    \end{align*}
\end{proof}

Given an operator $S\in \bdd$ and $O\subset G$, we formally define $S_O$ by the weak-integral
\begin{align}\label{eq:SO}
    S_O=c_\nu^2\int_{G\times G} \chi_{O}(x) \ipnu{S\pinu(y)1}{\pinu(yx)1}\  (\pinu(yx)1\otimes \pinu(y)1) \ d\mu_{G\times G}(x,y).
\end{align}
It is important to observe that the above integral with respect to the product measure may not exist. The definition of $S_O$ is motivated by Lemma \ref{lem:integral_S}. In the case where $O=G$, if the integral $S_G$ exists as a bounded operator, then $S_G$ equals $S$. This will be proved later in Proposition \ref{prop:S_G}.

For $\rho\geq 0$, we make the identification $\Br:=\Br(0)$.
Our goal is to prove the following:
\begin{enumerate}
    \item For any $S\in \bdd$, $S_{\Br}$ exists and $S_{\Br}\in \overline{\toep}$.
    \item For $S\in \wloc$,  $\lim_{\rho\to \infty} S_{\Br}=S$ in operator norm.
\end{enumerate}
This implies our main theorem.

\subsection{Approximations by Toeplitz operators}
Here, we provide the first part of our proof, namely, we show that for any $S\in \bdd$, the operator $S_{\Br}$ exists and can be approximated by Toeplitz operators. For this, we require several lemmas.
In the lemma below, we prove that $S_O$ exists whenever $O$ is of finite Haar measure.
\begin{lemma}\label{lem:finitemeasure}
    Let $S\in \bdd$ and let $O\subset G$ s.t. $\mu_G(O)<\infty$. Then $S_{O}$ exists as a bounded operator. Moreover, we can express $S_O$ as
    $$S_O=c_\nu^2\int_O a_{S,x} \ast \Phi_{x} \haar{x}$$
    where $a_{S,x}\in \bddfG$ is given by
    $$a_{S,x}(y)=\ipnu{S\pinu(y)1}{\pinu(y)\knu_{x\cdot 0}},\ \ y\in G$$
    and $\Phi_x=\knu_{x\cdot 0}\otimes 1$, for $x\in \Omega$.
\end{lemma}

\begin{proof}
    First, we show that $S_O$ exists.  For $f_1,f_2\in\Berg$, we have
    \begin{align*}
        \int_O\int_G  |\ipnu{S\pinu(y)1}{\pinu(yx)1}\ & \ipnu{(\pinu(yx)1\otimes \pinu(y)1)f_1}{f_2}| \haar{y}\haar{x}\\
         &\leq \|S\| \int_O\int_G |\ipnu{f_1}{\pinu(y)1}\ipnu{\pinu(yx)1}{f_2}| \haar{y}\haar{x}\\
         &=\|S\| \int_O\int_G |\ipnu{f_1}{\pinu(y)1}\ipnu{\pinu(y)\pinu(x)1}{f_2}| \haar{y}\haar{x}\\
         &\hspace{5cm} \text{(as $m(x,y)\in \T$)}\\
         &\leq \int_O \|\pi^\nu_{f_1,1}\|_2 \|\pi^\nu_{f_2,\pinu(x)1}\|_2 \haar{x}\quad \text{(by Holder's inequality)}\\
        &= \frac{1}{d_{\pinu}}\int_O \normnu{f_1}\normnu{1}\normnu{f_2}\normnu{\pinu(x)1}  \haar{x}\\
        &\hspace{5cm} \text{(by Schur-orthogonality relations \ref{eq:Schur})}\\
        &\leq\frac{1}{d_{\pinu}}\normnu{f_1}\normnu{f_2} \mu_G(O).
    \end{align*}
    Since $d_{\pinu}=c_\nu$, the operator $S_O$ exists with $\|S_O\|\leq c_\nu \mu_G(O)$.
    Also, note that
    \begin{align*}
        \ipnu{S\pinu(y)1}{\pinu(yx)1}\  &(\pinu(yx)1\otimes \pinu(y)1) \\
        &= \frac{1}{|m(y,x)|^2}\ipnu{S\pinu(y)1}{\pinu(y)\pinu(x)1}\  (\pinu(y)\pinu(x)1\otimes \pinu(y)1) \\
        &= \ipnu{S\pinu(y)1}{\pinu(y)\knu_{x\cdot 0}}\  (\pinu(y)\knu_{x\cdot 0}\otimes \pinu(y)1)\\
        &= a_{S,x}(y)\ \pinu(y)(\knu_{x\cdot 0}\otimes1)\pinu(y)^*\\
        &= a_{S,x}(y) \actopL{y}{\Phi_x}.
    \end{align*}
    Then it follows that 
    $$S_O=c_\nu^2\int_O a_{S,x} \ast \Phi_{x} \haar{x}.$$
\end{proof}

\begin{lemma}\label{lem:continuity_G}
    The map $x\mapsto a_{S,x} \ast \Phi_{x}$ from $G$ to $\bdd$ is continuous with respect to the operator norm.
\end{lemma}
\begin{proof}
    Note that for $x_1,x_2\in G$,
    \begin{align*}
        \|a_{S,x_1} \ast \Phi_{x_1}-a_{S,x_2} \ast \Phi_{x_2}\|
        &\leq \|a_{S,x_1} \ast (\Phi_{x_1}-\Phi_{x_2})\|+\|(a_{S,x_1}-a_{S,x_2}) \ast \Phi_{x_2}\|\\
        &\leq \frac{1}{d_{\pinu}}\|a_{S,x_1}\|_\infty \|\Phi_{x_1}-\Phi_{x_2}\|_\tr+\frac{1}{d_{\pinu}}\|a_{S,x_1}-a_{S,x_2}\|_\infty\|\Phi_{x_2}\|_\tr
    \end{align*}
    by QHA Young's inequality.
    Since $\knu_{x\cdot 0}=\pinu(x)1$, we have that
    $$\|\Phi_{x_1}-\Phi_{x_2}\|_\tr=\|(\pinu(x_1)1-\pinu(x_2)1)\otimes 1\|_\tr\leq \normnu{\pinu(x_1)1-\pinu(x_2)1}.$$
    Also, for $y\in G$,
    \begin{align*}
        |a_{S,x_1}(y)-a_{S,x_2}(y)|&=|\ipnu{S\pinu(y)1}{\pinu(y)(\pinu(x_1)1-\pinu(x_2)1)}|\\
        &\leq \|S\|\normnu{\pinu(x_1)1-\pinu(x_2)1}\\
        &=\|S\|\normnu{\pinu(x_2^{-1}x_1)1-1}.
    \end{align*}
    Therefore,
    \begin{align*}
        \|a_{S,x_1} \ast \Phi_{x_1}-a_{S,x_2} \ast \Phi_{x_2}\|& \leq \frac{2}{d_{\pinu}}\|S\|\normnu{\pinu(x_2^{-1}x_1)1-1}.
    \end{align*}
    The proof is now complete by the strong continuity of the representation.
\end{proof}

We warn the reader of a slight abuse of notation in defining the function $a_{S,z}$ and the operator $\Phi_z$ for $z\in \Omega$ in the lemma below. Since  for $x\in G$, the function $a_{S,x}$ and the operator $\Phi_x$, depend only on $x\cdot 0$, we may label them as $a_{S,z}$ and $\Phi_z$ with $z\in \Omega$.

\begin{lemma}\label{lem:S_Br}
    Let $S\in \bdd$ and $\rho\geq 0$. Then $S_{\Br}$ exists as a bounded operator. Moreover, we can express $S_{\Br}$ as
    $$S_{\Br}=c_\nu^2 \int_{\Dr(0)} a_{S,z} \ast \Phi_z \inv{z}$$
    where $a_{S,z}\in \bddfG$ is given by
    $$a_{S,z}(y)=\ipnu{S\pinu(y)1}{\pinu(y)k_{z}^\nu},\ \ y\in G$$
    and $\Phi_z=k_{z}^\nu\otimes 1$.
\end{lemma}
\begin{proof}
    We have $\mu_G(\Br)<\infty$ as $\Br$ is a precompact subset of $G$. Hence by Lemma \ref{lem:finitemeasure}, $S_{\Br}$ exists.
    Since $\pinu(x)1=k_{x\cdot 0}$, we have that
    \begin{align*}
        S_{\Br}&=c_\nu^2\int_{\Br} a_{S,x\cdot 0} \ast \Phi_{x\cdot 0} \haar{x}\\
        &=c_\nu^2 \int_{\Dr(0)} a_{S,z} \ast \Phi_z \inv{z}
    \end{align*}
    by \ref{eq:integral}.
\end{proof}

\begin{lemma}\label{lem:continuity_D}
    The map $z\mapsto a_{S,z} \ast \Phi_{z}$ from $\Omega$ to $\bdd$ is continuous with respect to the operator norm.
\end{lemma}

\begin{proof}
    We recall that $\Omega$ is the homogeneous space $G/K$. Let $i:G\to \Omega$ be the projection map. 
    Define $F:\Omega\to \bdd$ by 
    $$F(z)=a_{S,z} \ast \Phi_{z}.$$
    Then $F$ is continuous if and only if $F\circ i:G\to \bdd$ is continuos. Note that
    $$(F\circ i)(x)=a_{S,x\cdot 0} \ast \Phi_{x\cdot 0}=a_{S,x} \ast \Phi_x,\quad x\in G.$$
    We have already verified the continuity of $F\circ i$ in Lemma \ref{lem:continuity_G}. This completes the proof.
\end{proof}

\begin{proposition}\label{prop:in_Toepalg}
    Let $S\in \bdd$ and let $\rho\in(0,\infty)$. Then   $S_{\Br}\in \overline{\toep}$. 
\end{proposition}

\begin{proof}
    By Lemma \ref{lem:convol_traceclass}, we have that for each $z\in \Omega$,
    $$a_{S,z}\ast \Phi_z\in \overline{\toep}.$$
    Now by Lemma \ref{lem:S_Br},
    \begin{align*}
        S_{\Br}=c_\nu^2 \int_{\Dr(0)} a_{S,z} \ast \Phi_z \inv{z}= c_\nu^2 \int_{\Dr(0)} \frac{1}{h(z,z)^p}(a_{S,z} \ast \Phi_z) \ dz.
    \end{align*}
    Furthermore, $\Dr(0)$ is a precompact subset of $\Omega$ and the map $z\mapsto \frac{1}{h(z,z)^p}(a_{S,z} \ast \Phi_z)$ from $\Omega$ to $\bdd$ is continuous with respect to the operator norm by Lemma \ref{lem:continuity_D} and because $h(z,z)$ is a polynomial in $z$ and $\Bar{z}$ vanishing nowhere in $\Omega$. Therefore, $S_{\Br}$ is merely a limit of the Riemann sums of operators of the form $\frac{1}{h(z,z)^\gen}(a_{S,z} \ast \Phi_z)$. This proves that 
    $$S_{\Br}\in \overline{\toep},$$
    concluding the proof.
\end{proof}

\subsection{An integral representation of weakly localized operators}
Here, we prove that a weakly localized operator $S\in \wloc$ has the integral representation $S=S_G$ and $\lim_{\rho\to \infty}S_{\Omega\setminus\Br}=0$ in operator norm.

\begin{proposition}\label{prop:S_G}
    Let $\nu>-1+(r-1)a/2$ and let $\alpha\in (\frac{\gen-\nu-2}{2},\frac{\nu+b+2}{2})$. Let $S\in\wloc$. Then
    $\lim_{\rho\to \infty} S_{\Br}=S$
    in operator norm.
\end{proposition}

\begin{proof}
    Let $f_1,f_2\in \Berg$. 
    Define the integral $I$ by 
    $$I= \int_{G\setminus \Br}\int_{G} |\ipnu{S\pinu(y)1}{\pinu(yx)1}  \ \ipnu{(\pinu(yx)1\otimes \pinu(y)1)f_1}{f_2}| \ \haar{y} \haar{x}.$$
    Now by Tonelli's theorem followed by a change of variable, we get
    \begin{align*}
        I&= \int_{G}\int_{G\setminus\Br(y)} |\ipnu{S\pinu(y)1}{\pinu(x)1}  \ \ipnu{(\pinu(x)1\otimes\pinu(y)1)f_1}{f_2}| \ \haar{x} \haar{y}\\
        &=\int_{G}\int_{G\setminus\Br(y)} |\ipnu{S\pinu(y)1}{\pinu(x)1}  \ipnu{f_1}{\pinu(y)1}\ipnu{\pinu(x)1}{f_2}| \ \haar{x} \haar{y}\\
        &=\int_{G}|\ipnu{f_1}{\pinu(y)1}|\int_{G\setminus\Br(y)} |\ipnu{S\pinu(y)1}{\pinu(x)1} \ipnu{\pinu(x)1}{f_2}| \ \haar{x} \haar{y}\\
        &\leq \|f_1\|\Bigg( \int_{G} \bigg(\int_{G\setminus \Br(y)} |\ipnu{S\pinu(y)1}{\pinu(x)1} \ipnu{\pinu(x)1}{f_2}| \haar{x}  \bigg)^2 \haar{y}\Bigg)^\frac{1}{2}\\
        &\hspace{4cm}\text{(by Holder's inequality)}\\
        &= \|f_1\| \Bigg( \int_{G} F(y)^2\haar{y}\Bigg)^\frac{1}{2}
    \end{align*}
    where
    $$F(y)=\int_{G\setminus \Br(y)} |\ipnu{S\pinu(y)1}{\pinu(x)1} \ipnu{\pinu(x)1}{f_2}| \haar{x}.$$
    Now note that
    \begin{align*}
        F(y)&=\int_{G\setminus \Br(y)}|\ipnu{S\pinu(y)1}{\pinu(x)1}|^\frac{1}{2} \frac{h(x,x)^\frac{\alpha}{2}}{h(y,y)^\frac{\alpha}{2}}\\
        &\hspace{4cm} \times |\ipnu{S\pinu(y)1}{\pinu(x)1}|^\frac{1}{2}| \ipnu{\pinu(x)1}{f_2}| \frac{h(y,y)^\frac{\alpha}{2}}{h(x,x)^\frac{\alpha}{2}}\haar{x}\\
        &\leq \Big(\int_{G\setminus \Br(y)} |\ipnu{S\pinu(y)1}{\pinu(x)1}| \frac{h(x,x)^\alpha}{h(y,y)^\alpha} \haar{x} \Big)^\frac{1}{2} \\
        &\hspace{4cm} \times \Big( \int_G |\ipnu{S\pinu(y)1}{\pinu(x)1}|| \ipnu{\pinu(x)1}{f_2}|^2 \frac{h(y,y)^\alpha}{h(x,x)^\alpha} \haar{x} \Big)^\frac{1}{2}\\
        &\leq \Ione{S}(\rho)^\frac{1}{2} \Big( \int_G |\ipnu{S\pinu(y)1}{\pinu(x)1}| |\ipnu{\pinu(x)1}{f_2}|^2 \frac{h(y,y)^{\alpha}}{h(x,x)^{\alpha}} \haar{x} \Big)^\frac{1}{2}.
    \end{align*}

    Therefore, we have
    \begin{align*}
        I&\leq  \Ione{S}(\rho)^\frac{1}{2}\|f_1\| \Bigg( \int_{G} \int_G |\ipnu{S\pinu(y)1}{\pinu(x)1}| |\ipnu{\pinu(x)1}{f_2}|^2 \frac{h(y,y)^{\alpha}}{h(x,x)^{\alpha}} \haar{x} \haar{y}\Bigg)^\frac{1}{2}\\
         &=  \Ione{S}(\rho)^\frac{1}{2}\|f_1\|  \Bigg( \int_{G} |\ipnu{\pinu(x)1}{f_2}|^2 \int_G |\ipnu{S^*\pinu(x)1}{\pinu(y)1}|  \frac{h(y,y)^{\alpha}}{h(x,x)^{\alpha}} \haar{y} \haar{x} \Bigg)^\frac{1}{2}\\
         &\leq \Ione{S}(\rho)^\frac{1}{2} \Ione{S^*}(0)^\frac{1}{2}\|f_1\| \Bigg( \int_{G} |\ipnu{\pinu(x)1}{f_2}|^2  \haar{x} \Bigg)^\frac{1}{2}\\
        &= \Ione{S}(\rho)^\frac{1}{2} \Ione{S^*}(0)^\frac{1}{2}\|f_1\| \|f_2\| \quad \text{(by \ref{eq:integral_nu} and the reproducing formula)}. 
    \end{align*}

    Then $S_{G\setminus \Br}$ exists for any $\rho\geq 0$ with $\|S_{G\setminus \Br}\|\leq c_\nu^2\Ione{S}(\rho)^\frac{1}{2} \Ione{S^*}(0)^\frac{1}{2}$. Since $S\in \wloc$, we have
    $$\lim_{\rho\to \infty}S_{G\setminus \Br}=0$$
    in operator norm.
    In particular $S_G\in \bdd$ and
    \begin{align*}
        S_G&=\int_{G}\int_{G} \ipnu{S\pinu(y)1}{\pinu(yx)1}  \ipnu{(\pinu(yx)1\otimes \pinu(y)1)f_1}{f_2} \ \haar{y} \haar{x}\\
        &=\int_{G}\int_{G} \ipnu{S\pinu(y)1}{\pinu(x)1}  \ \ipnu{(\pinu(x)1\otimes \pinu(y)1)f_1}{f_2} \ \haar{x} \haar{y}\\
        &\hspace{4cm}\text{(by Fubini's theorem and a change of variable)}\\
        &= S \quad \text{(by Lemma \ref{lem:integral_S})}.
    \end{align*}
    Hence $S=S_G=S_{\Br}+S_{G\setminus \Br}$.
    Now we can conclude that the following convergence holds in the operator norm:
    $$\lim_{\rho\to \infty}S_{\Br}=S.$$
    \end{proof}

    Now it is evident that the following inclusion holds by propositions \ref{prop:in_Toepalg} and \ref{prop:S_G}.

    \begin{proposition}
        Let $\nu>-1+(r-1)a/2$ and let $\alpha\in (\frac{\gen-\nu-2}{2},\frac{\nu+b+2}{2})$. Then the following inclusion holds:
        $$\wloc\subset \overline{\toep}.$$
    \end{proposition}

    Now we have the required description of the Toeplitz algebra.
    \begin{theorem}\label{theo:wloc_Toeplitz}
        Let $\nu>-1+(r-1)a/2$ and let $\alpha\in (\frac{\gen-\nu-2}{2},\frac{\nu+b+2}{2})$. We have the following
        $$\toepalg=\overline{\wloc}=\overline{\toep}=\overline{\bddfG\ast \traceop}.$$
    \end{theorem}

    \begin{proof}
        We have proved the inclusions:
        $$\toep\subset \wloc\subset \overline{\toep}$$
        and from lemmas \ref{lem:Toeplitz} and \ref{lem:convol_traceclass}, we have
        $$\toep\subset \bddfG\ast \traceop\subset \toepalg.$$
        Furthermore $\overline{\wloc}$ is a $C^*$-algebra. Then by the minimality of the Toeplitz algebra $\toepalg$ and by the above inclusions, we get the required equalities.
    \end{proof}

    \subsection{Discussion}
    While we now have that Toeplitz operators are norm dense in the Toeplitz algebra of certain weighted Bergman spaces $\Berg$ of any bounded symmetric domain, we still do not have a complete characterization of the Toeplitz algebra that does not involve taking the closure. As for the Fock space, using QHA, Fulsche characterized the Toeplitz algebra as the "algebra of all bounded uniformly continuous operators." This characterization makes analysis of the Toeplitz algebra on the Fock space much more agreeable. Now we pose the following questions for bounded symmetric domains.
    \begin{enumerate}
        \item Can we characterize the Toeplitz algebra $\toepalg$ as an algebra of "uniformly continuous operators" with some notion of uniform continuity?
        \item In \cite{H17} and \cite{H19}, Hagger discusses band-dominated operators and proves that Toeplitz operators are band-dominated. Does the algebra of band-dominated operators coincide with the Toeplitz algebra?
        \item Is the algebra of $\alpha$-weakly-localized operators closed?
        \item We require the extra restriction $\nu>-1+(r-1)a/2$. Is it possible to generalize these results to $\nu>-1$?
        \item Lastly, does the norm density of Toeplitz operators in the Toeplitz algebra hold for any complex domain $\Omega\subset\C^n$? 
    \end{enumerate}

    \section{$H$- invariant Toeplitz algebras}\label{sec:HToeplitz}
    Toeplitz algebras generated by Toeplitz operators that are invariant under group actions, is a widely explored topic in operator theory. Particularly because such invariances may lead to commutative Toeplitz algebras. We refer to \cite{BV12, BV13, GQV06, Q16, QV07,QV08} for discussions on commutative Toeplitz algebras, specializing to the Bergman space over the unit ball $\cA^2(\B^n)$. In \cite{DOQ15}, the authors discuss commutative Toeplitz algebras over $\Berg$ of a bounded symmetric domain $\Omega$. 
    
    The density of Toeplitz operators in the radial Toeplitz algebra over $\cA^2(\B^n)$ was proved in \cite{S08}, using the notion of Laplacian of an operator. This result was then used in \cite{GMV13}, to discuss the Gelfand theory of the radial Toepliz algebra over $\cA^2(\B^n)$. In \cite{DM23}, it was observed that Toeplitz operators are dense in the separately radial Toeplitz algebra over $\cA^2(\B^n)$. This was achieved by using the average over the representation, along with the density of Toeplitz operators in the full Toeplitz algebra. This method can be used whenever the group in consideration is compact.
    
     In this section, we note that compactness of the group is not necessary. We establish an $H$-invariant version of Theorem \ref{theo:wloc_Toeplitz}, for $H<G$. Thus the density of Toeplitz operators in the Toeplitz algebra, carry over to $H$-invariant Toeplitz algebras over $\Berg$ as well.
    We will always assume $H<G$ is a closed subgroup. 

    \subsection{$H$-invariant functions and operators}
    A function $f\in \bddfG$ is said to be $H$-invariant if 
    $$\actfL{h}{f}=f,\quad h\in H.$$
     An operator $S\in \bdd$ is said to be $H$-invariant if 
     $$\actopL{h}{S}=S,\quad h\in H.$$
     We denote the set of all $H$-invariant elements in $\bddfG$ by $\bddfG^H$, and the set of all $H$-invariant bounded operators is denoted by $\bddH$. Then $\bddfG^H$ and $\bddH$ form closed subalgebras of $\bddfG$ and $\bdd$ respectively. 

     \begin{lemma}\label{lem:translation}
         Let $a\in \bddfG$, $\psi\in \Lone$, and $S\in \traceop$. Then for $x\in G$,
         \begin{enumerate}
             \item $\actfL{x}{a\ast \psi}=\actfL{x}{a}\ast \psi$
             \item $\actopL{x}{a\ast S}=\actfL{x}{a}\ast S$.
         \end{enumerate}
         Moreover, if $a\in \bddfG^H$, then $a\ast \psi\in \bddfH$ and $a\ast S\in \bddH$.
     \end{lemma}
     \begin{proof}
         As the proofs of (1) and (2) are similar, we only prove statement (2). Note that for $x\in G$,
         \begin{align*}
             \actopL{x}{a\ast S}&= \int_{G} a(y) \actopL{x}{\actopL{y}{S}} \haar{y}\\
             &=\int_{G} a(y)  \actopL{xy}{S} \haar{y} \quad \text{(by \ref{eq:homomorph})}\\
             &=\int_G a(x^{-1}y) \actopL{y}{S} \haar{y}\quad \text{(by a change of variable)}\\
             &=\actfL{x}{a}\ast S.
         \end{align*}
         Now the second claim is immediate.
     \end{proof}
     Denote the set of all $H$-invariant elements in $\bddf$ by $\bddfH$. The following proposition characterizes $H$-invariant Toeplitz operators.
     \begin{proposition}\label{prop:Toeplitzcri}
         Let $a\in \bddf$. Then $T_a\in \bddH$ if and only if $a\in \bddfH$.
     \end{proposition}
     \begin{proof}
         We have that $T_a=a\ast \Phi\in \bddH$ if $a\in \bddfH$ by Lemma \ref{lem:translation}. 
         To prove the converse, assume $T_a\in \bddH$. Then by Lemma \ref{lem:translation},
         $$\actfL{h}{a}\ast \Phi=\actopL{h}{a\ast \Phi}=a\ast \Phi,\quad h\in H.$$
         Then since the map $a\mapsto T_a$ is injective (see subsection \ref{subsec:Berg}), we get that
         $$\actfL{h}{a}=a,\quad h\in H.$$
         Therefore, $a\in \bddfH$.
     \end{proof}

     Let $\toepalgH$ denote the subalgebra of $\toepalg$ generated by the set of all Toeplitz operators $T_a$, with $a\in \bddfH$. Also, define the set $\toep^H$ by
     $$\toep^H:=\{T_a\in \bdd\mid a\in \bddfH\}.$$
     The following is an $H$-invariant analog of Lemma \ref{lem:convol_traceclass}.
     
     \begin{lemma}\label{lem:Hconvol_trace}
         We have the following:
         \begin{align*}
         \bddfG^H\ast \traceop\subset \overline{\toep^H}.
     \end{align*}
     \end{lemma}
     \begin{proof}
         Let $S\in \bddfG^H\ast \traceop$. Then $S=a\ast T$, for some $a\in \bddfG^H$ and $T\in \traceop$.  Now by an argument similar to the proof of Lemma \ref{lem:convol_traceclass}, $S$ can be approximated by Toeplitz operators of the form
         $$(a\ast \psi)\ast \Phi, \quad \psi\in \Lone.$$
         We note that $a\ast \psi$ is $H$-invariant as, for $h\in H$,
         $$\actfL{h}{a\ast \psi}=\actfL{h}{a}\ast \psi=a\ast \psi,$$
         concluding the proof.
     \end{proof}
     \subsection{Weakly localized operators and the $H$-invariant Toeplitz algebra}
     Let $\nu>-1+(r-1)a/2$.
     For $\alpha\in (\frac{\gen-\nu-2}{2},\frac{\nu+b+2}{2})$, let
     $$\wloc^H:=\wloc\cap \bddH.$$
     Then $\wloc^H$ is a selfaljoint algebra whose norm closure is a $C^*$-subalgebra of $\bddH$. Furthermore, by propositions \ref{prop:Toeplitz_wloc} and \ref{prop:Toeplitzcri}, we have the inclusion
        \begin{align}\label{eq:HToeplitz_in}
            \{T_a\mid a\in \bddfH\}\subset \wloc^H.
        \end{align}
    The following lemma is the last ingredient we need.
    
    \begin{lemma}\label{lem:Hwloc_in}
         Let $\alpha\in (\frac{\gen-\nu-2}{2},\frac{\nu+b+2}{2})$. Then
         $$\wloc^H\subset \overline{\toep^H}.$$
    \end{lemma}
    \begin{proof}
        Let $S\in \wloc^H$. Then by Proposition \ref{prop:S_G}, $\lim_{\rho\to \infty} S_{\Br}=S$ in operator norm. Hence, by Lemma \ref{lem:Hconvol_trace}, it is enough to prove that
        $$ S_{\Br}\in \overline{\bddfG^H\ast \traceop}.$$
        Note that by Lemma \ref{lem:S_Br},
        $$S_{\Br}=c_\nu^2 \int_{\Dr(0)} a_{S,z} \ast \Phi_z \inv{z}.$$
    Note that for $h\in H$ and $x\in G$,
    \begin{align*}
        \actfL{h}{a_{S,z}}(x)&= a_{S,z}(h^{-1}x)\\
        &=\ipnu{S\pinu(h^{-1}x)1}{\pinu(h^{-1}x)k_{z}^\nu}\\
        &=\ipnu{\actopL{h}{S}\pinu(x)1}{\pinu(x)k_z^\nu}\\
        &=\ipnu{S\pinu(x)1}{\pinu(x)k_z^\nu}\quad \text{(as $S$ is $H$-invariant)}\\
        &=a_{S,z}(x).
    \end{align*}
    Therefore $a_{S,z}\ast \Phi_z\in \bddfG^H\ast \traceop.$ Then we have
    $$S_{\Br}\in \overline{\bddfG^H\ast \traceop}$$
    by a similar argument as in the proof of Proposition \ref{prop:in_Toepalg}.
        
    \end{proof}
     \begin{theorem}\label{theo:wloc_ToeplitzH}
        Let $\nu>-1+(r-1)a/2$ and let $\alpha\in (\frac{\gen-\nu-2}{2},\frac{\nu+b+2}{2})$. We have the following:
        $$\toepalgH=\overline{\wloc^H}=\overline{\toep^H}=\overline{\bddfG^H\ast \traceop}.$$
    \end{theorem}
    \begin{proof}
        The proof of is similar to the proof of Theorem \ref{theo:wloc_Toeplitz} and uses \ref{eq:HToeplitz_in} and Lemmas \ref{lem:Hconvol_trace} and \ref{lem:Hwloc_in}.
    \end{proof}

    \noindent \textbf{Acknowledgements:} The author acknowledges Gestur \'Olafsson, Mishko Mitkovski and Matthew Dawson for many useful discussions. The author is also thankful to the anonymous reviewer for suggestions that helped her improve the article.\\

    \noindent \textbf{Declaration of generative AI and AI-assisted technologies in the writing process:}
    During the preparation of this work, the author did not use any generative AI and AI-assisted technologies.

\bibliographystyle{plain} 
\bibliography{main}

\end{document}